\documentclass[11pt]{article}
\title{Graded Clifford algebras of prime global dimension with an action of $H_p$}
\author{De Laet, Kevin}
\date{}
\usepackage{amsmath,amsfonts,array,amscd,amsthm,amssymb,stackrel,verbatim,wrapfig,subfig,float}
\usepackage{enumerate}
\usepackage{url}
\usepackage{tikz}
\usepackage[all]{xy}
\usetikzlibrary{arrows,matrix}
\usetikzlibrary{shapes.geometric}
\theoremstyle{plain}
\newtheorem{theorem}{Theorem}[section]
\newtheorem{lemma}[theorem]{Lemma}
\newtheorem{proposition}[theorem]{Proposition}

\newtheorem{remark}[theorem]{Remark}
\newtheorem{example}[theorem]{Example}
\newtheorem{definition}[theorem]{Definition}
\DeclareMathOperator{\Ext}{Ext}
\usepackage{natbib}
\numberwithin{equation}{section}
\begin{document}
\maketitle
\begin{abstract}
\noindent In this article we study graded Clifford algebras with a gradation preserving action of automorphisms given by $H_p$, the Heisenberg group of order $p^3$ with $p$ prime. After reviewing results in dimensions $3$ and $4$, we will determine the graded Clifford algebras that are $AS$-regular algebras of global dimension 5 and generalise certain results to arbitrary dimension $p$.
\end{abstract}
\section{Introduction}
In \cite{artin1987graded}, Artin and Schelter classified the Artin-Schelter regular algebras of global dimension 3, that is, graded associative algebras over $\mathbb{C}$ with excellent homological properties. The algebras of type A with 3 generators in this classification are the Sklyanin algebras, with defining relations
\begin{align*}
axy+byx+cz^2=0,\\
ayz+bzy+cx^2=0,\\
azx+bxz+cy^2=0,
\end{align*}
with some restrictions on $(a:b:c) \in \mathbb{P}^2$. In \cite{artin2007some}, Artin, Tate and Van den Bergh showed that these Sklyanin algebras depend on an elliptic curve $E$ and a point $\tau \in E$. It was noticed by Smith and Tate in \cite{smith1994center} that the Heisenberg group $H_3$ of order 27 acts on such an algebra as gradation preserving automorphisms. In \cite{odesskii1989sklyanin}, Odesskii and Feigin generalised the Sklyanin algebras to every dimension $n\geq 3$ and defined the $n$-dimensional Sklyanin algebras $A_n(\tau,E)$, where for each $n$ $H_n$ again works as algebra automorphisms on $A_n(\tau,E)$. In those cases, $A_n(\tau,E)_1 \cong V = \mathbb{C}x_0 + \ldots +\mathbb{C}x_n$ as an $H_n$-representation, with the action of $H_n$ given by
\begin{displaymath}
e_1 \cdot x_i = x_{i-1}, e_2 \cdot x_i= \omega^i x_i, 
\end{displaymath}
$\omega$ being a primitive $n$th root of unity and the indices taken $\bmod n$. Since the Heisenberg group of order $n^3$ plays such an important role in the study of these Sklyanin algebras, other examples of graded Artin-Schelter regular algebras with $n$ generators and on which $H_n$ acts as gradation preserving automorphisms would be desired and in the best case, a complete classification. While this classification has been made for $n=3$, a classification for $n \geq 4$ is not yet found.
\par Another interesting class of algebras is given by graded Clifford algebras, i.e. algebras depending on a quadratic form over a polynomial ring in an indeterminate number of variables, with entries of degree 2. Such algebras are always finite over their center and therefore they have a rich representation theory.
\par This paper looks at algebras that belong to these two worlds: it will discuss graded algebras $\mathfrak{C}(a,b),(a,b) \in \mathbb{A}^2$ ($\mathfrak{C}(A:B:C), (A:B:C) \in \mathbb{P}^2$), with generators in degree 1 and relations of the form
\begin{align*}
x_1x_4+x_4x_1 = ax_0^2,&& x_2x_3+x_3x_2 = bx_0^2,\\
x_2x_0+x_0x_2 = ax_1^2,&& x_3x_4+x_4x_3 = bx_1^2,\\
x_3x_1+x_1x_3 = ax_2^2,&& x_4x_0+x_0x_4 = bx_2^2,\\
x_4x_2+x_2x_4 = ax_3^2,&& x_0x_1+x_1x_0 = bx_3^2,\\
x_0x_3+x_3x_0 = ax_4^2,&& x_1x_2+x_2x_1 = bx_4^2.
\end{align*}
We will call these $H_5$-Clifford algebras (although these algebras aren't always Clifford algebras). Generically, a $H_5$-Clifford algebra $\mathfrak{C}(a,b)$ will be a graded Clifford algebra generated in degree 1, with associated quadratic form
\begin{equation}
\begin{bmatrix}
2x_0^2  & bx_3^2  & a x_1^2 & a x_4^2 & b x_2^2 \\
b x_3^2 & 2x_1^2  & b x_4^2 & a x_2^2 & a x_0^2 \\
a x_1^2 & b x_4^2 & 2 x_2^2 & b x_0^2 & a x_3^2 \\
a x_4^2 & a x_2^2 & b x_0^2 & 2 x_3^2 & b x_1^2 \\
b x_2^2 & a x_0^2 & a x_3^2 & b x_1^2 & 2 x_4^2
\end{bmatrix}
\end{equation}
on the polynomial ring $\mathbb{C}[x_0^2,x_1^2,x_2^2,x_3^2,x_4^2]$. $\mathfrak{C}(a,b)_1$ will be isomorphic as an $H_5$-representation to $V_1$, where $V_1$ is the unique simple representation such that for $z=[e_1,e_2], \varphi(z) = \omega I_5$ with $\omega = e^{\frac{2\pi i}{5}}$ and $\varphi:H_5 \rightarrow \mathbf{M}_5(\mathbb{C})$ the morphism determined by $V_1$. Crucial to describing these algebras will be the Koszul dual $\mathfrak{C}(a,b)^!$, which will be a commutative algebra with 5 generators and 5 homogeneous relations and therefore it will define a projective variety in $\mathbb{P}^4$.
\par Sometimes we will make the 10 relations we are interested in homogeneous, i.e. we will look at the algebras with equations
\begin{align*}
C(x_1x_4+x_4x_1) = Ax_0^2, &&C(x_2x_3+x_3x_2) = Bx_0^2,\\
C(x_2x_0+x_0x_2) = Ax_1^2, &&C(x_3x_4+x_4x_3) = Bx_1^2,\\
C(x_3x_1+x_1x_3) = Ax_2^2, &&C(x_4x_0+x_0x_4) = Bx_2^2,\\
C(x_4x_2+x_2x_4) = Ax_3^2, &&C(x_0x_1+x_1x_0) = Bx_3^2,\\
C(x_0x_3+x_3x_0) = Ax_4^2, &&C(x_1x_2+x_2x_1) = Bx_4^2.
\end{align*}
and the following relations, which come naturally if $C=1$, but must be included when $C = 0$.
\begin{align*}
B(x_1x_4+x_4x_1) = A(x_2x_3+x_3x_2),\\
B(x_2x_0+x_0x_2) = A(x_3x_4+x_4x_3),\\
B(x_3x_1+x_1x_3) = A(x_4x_0+x_0x_4),\\
B(x_4x_2+x_2x_4) = A(x_0x_1+x_1x_0),\\
B(x_0x_3+x_3x_0) = A(x_1x_2+x_2x_1).
\end{align*}
Considering this, if we put $C = 0$, we still get an algebra with 10 quadratic relations, but we cannot hope to find Artin-Schelter regular algebras of global dimension 5 this way, since in these cases the algebras will not be of finite global dimension.
\par In the last section we will generalise certain phenomena in dimension 5 to $H_p$-Clifford algebras, $p$ prime. The main result regarding quantum spaces will be
\begin{theorem}
There are exactly $p+1$-points (corresponding to the points in $\mathbb{P}^1_{\mathbb{F}_p}$) in $\mathbb{P}^{\frac{p-1}{2}}$ for which the corresponding $H_p$-Clifford algebra will be isomorphic to the quantum space $\mathbb{C}_{-1}[x_0,\ldots,x_{p-1}]$.
\end{theorem}
Considering the quantum spaces, we also have the following duality
\begin{theorem}
There is a 1-to-1 correspondence between the $PSL_2(p)$-orbit of the point $(1:0:\ldots:0)$ and the $PSL_2$-orbit of the line $a_0 = 0$. This correspondence is a morphism of $PSL_2(p)$-sets and its action coincides with the canonical action of $PSL_2(p)$ on $\mathbb{P}^1_{\mathbb{F}_p}$.
\end{theorem}
The main result regarding the regular $H_p$-Clifford algebras will be
\begin{theorem}
The character series of a regular $H_p$-Clifford algebra is the same as the character series of the polynomial ring $\mathbb{C}[x_0,\ldots,x_{p-1}]$, with the degree 1 part of the $H_p$-Clifford algebra isomorphic to the degree 1 part of $\mathbb{C}[x_0,\ldots,x_{p-1}]$ as an $H_p$-representation.
\end{theorem}
This means that, as an $H_p$-module, a $H_p$-Clifford algebra cannot be distinguished from the polynomial ring in $p$ variables.
\subsection*{Notations}
Throughout the paper, we will use the following notations and conventions
\begin{itemize}
\item $\mathbb{C}_{-1}[y_1,\ldots,y_k]$ is the noncommutative algebra with defining relations $y_i y_j = -y_j y_i, 1 \leq i < j \leq k$.
\item For a subset $S \subset \mathbb{C}[y_1,\ldots,y_k]$, $\mathbf{V}(S)$ is the Zariski-closed subset defined by the elements of $S$. It will be clear from the context whether we look at subsets of $\mathbb{A}^k$ or $\mathbb{P}^{k-1}$.
\item $\mu_n$ is the set of $n$th roots of unity in $\,\mathbb{C}$.
\item If $G$ is a finite group and $V$ is a $G$-representation, then the associated character of $G$ is denoted by $\chi_V$.
\end{itemize}
\section*{Acknowledgements}
The author would like to thank M. Van den Bergh, who proposed to look at these algebras and who has been a tremendous help along the way and L. le Bruyn, who gave ideas on what to do with these algebras and who also gave tips on how to write a readable paper. The author would also like to thank T. Raedschelders for suggesting to include a section with preliminaries.
\section{Preliminaries}
Those readers that are familiar with Koszul algebras, graded Clifford algebras, the representation theory of the Heisenberg group $H_p$ with $p$ prime and/or the connection between modular curves $X(p)$ and $H_p$ may skip this section. The last part of the subsection Koszul algebras however will be unfamiliar to most readers.
\subsection{Koszul algebras}
\begin{definition}
Given a quadratic algebra $A = T(V)/I$ with generators $V = \mathbb{C}x_0+\ldots+\mathbb{C}x_n$ and relations given by $I_2$, we define the Koszul dual to be the quadratic algebra $T(V^*)/J$, with $J_2$ defined as the subspace of $V^* \otimes V^*$ such that $\forall w \in J_2, \forall v \in I_2: w(v) = 0$.
\end{definition}
We say that $A$ is Koszul iff $A^! \cong \Ext_A(\mathbb{C},\mathbb{C})$. The standard properties of Koszul algebras we will need are that there is a relation between the Hilbert series of $A$ and $A^!$, given by
\begin{displaymath}
H_A(t)H_{A^!}(-t) = 1
\end{displaymath}
and that $A$ is Koszul iff $A^!$ is Koszul.
\par An important fact concerning Koszul algebras is that the Koszul complex associated to these algebras is of the form
\begin{displaymath}
\xymatrix{ \ldots \ar[r] & A \otimes (A^!)_n^* \ar[r]^-{(d_K)_n} & A \otimes (A^!)_{n-1}^* \ar[r] & \ldots}
\end{displaymath}
and that 
\begin{displaymath}
(A^!)_n^* = V^{\otimes n-2} \otimes I_2 \cap \ldots \cap I_2 \otimes V^{\otimes n-2}.
\end{displaymath}
$(d_K)_n$ is given by taking the first component of $(A^!)_n^*$ and absorbing it in $A$, for example $(d_K)_1(a \otimes x) = ax \in A$. It follows from this description that each $(d_K)_n$ is a $G$-morphism, whenever $G$ acts on $A$ as gradation preserving algebra automorphisms. This is useful for finding the character series when $G$ is a finite (or more generally, reductive) group.
\begin{definition}
Let $G$ be a finite group. The character series for an element $g \in G$ and for a graded algebra $A$ on which $G$ acts as gradation preserving automorphisms is a formal sum 
\begin{displaymath}
Ch_A(g,t) = \sum_{n \in \mathbb{Z}} \chi_{A_n}(g) t^n.
\end{displaymath}
\end{definition}
For example, if $g =  1$, $Ch_A(1,t) = H_A(t)$, the Hilbert series of $A$. Since a character of a representation is constant on conjugacy classes, we can represent the decomposition of $A$ in simple $G$-representations as a vector of length equal to the number of conjugacy classes and in the $i$th place the character series $Ch_A(g,t)$ with $g \in C_i$, the $i$th conjugacy class.
\par Suppose now that $A$ is a Koszul algebra and that a finite group $G$ acts on it as gradation preserving automorphisms. Because the Koszul complex is a free resolution of the trivial module $\mathbb{C}$, which is isomorphic as $G$-representation to the trivial representation and because the Koszul complex consists of $G$-morphisms, we have a similar formula for finding the character series of the Koszul dual as we have for the Hilbert series. More precisely, we have
\begin{align}
Ch_A(g,t) Ch_{(A^!)^*}(g,-t) = 1.
\label{al:chKos}
\end{align}
This allows us to compute $Ch_{A^!}(g,t)$ whenever we know $Ch_A(g,t)$. To know the character series of $A^!$, we have to take the complex conjugates of the coefficients of $Ch_{(A^!)^*}(g,t)$.
\subsection{Graded Clifford algebras}
This subsection will deal with the case we are interested in, but this is not the general definition. For more information, see \cite{LeBruyn1994d} or \cite{vancliff1998some}.
\par Given our algebra $\mathfrak{C}(a,b)$, we can associate to it 5 quadratic equations in the following way: we can write our equations as $x_i x_j + x_j x_i = (M_k)_{ij}x_k^2$ with $M_k \in \mathbf{M}_5(\mathbb{C})$. Taking $\{z_0,\ldots,z_4\}$ as the basis of $(\mathfrak{C}(a,b)_1)^*$ such that $z_i(x_j)=\delta_{ij}$, we get 5 quadratic equations
\begin{displaymath}
q_k=[z_0,z_1,z_2,z_3,z_4] M_k \begin{bmatrix}
z_0 \\ z_1 \\ z_2\\z_3\\z_4
\end{bmatrix}, k = 0,\ldots ,4
\end{displaymath}
This way, we get a quadric system. $\cap_{i=0}^4\mathbf{V}(q_i)$ defines a Zariski closed set in $\mathbb{P}(\mathfrak{C}(a,b)_1)$ and a point in this closed set is called a base point. It is also clear that this set parametrizes the degree 1 elements of $\mathfrak{C}(a,b)$ whose square is 0, this follows from the definition of a Clifford algebra. The algebra $\mathbb{C}[z_0,z_1,z_2,z_3,z_4]/(q_0,q_1,q_2,q_3,q_4)$ is the Koszul dual $\mathfrak{C}(a,b)^!$ of $\mathfrak{C}(a,b)$. 
\par In our case, the matrix $M_0$ for example is given by
\begin{displaymath}
M_0 = \begin{bmatrix}
2&0&0&0&0\\
0&0&0&0&a\\
0&0&0&b&0\\
0&0&b&0&0\\
0&a&0&0&0
\end{bmatrix}
\end{displaymath}
and it follows that $q_0 = 2z_0^2+2az_1z_4+2bz_2z_3$. Since our relations are Heisenberg invariant, the other $q_k, k = 1,\ldots 4$ are easily found by cyclic permutation of the indices.
\par The next theorem will be crucial.
\begin{theorem}[\cite{cassidy2010generalizations}]
A graded Clifford algebra $\mathfrak{C}$ is quadratic, Auslander-regular of global dimension $n$, satisfies the Cohen-Macaulay property and has as Hilbert series $\frac{1}{(1-t)^n}$ if and only if the associated quadric system is base-point free. If this is the case, $\mathfrak{C}$ is also a noetherian domain and Artin-Schelter regular.
\label{th:Reg}
\end{theorem}
\begin{remark}
While we will mainly work with graded Clifford algebras, we will also study quadratic algebras that are not Clifford algebras. These algebras will have as a quotient a graded Clifford algebra, the extra relations being implied by the fact that every square of an element of degree 1 is central in a graded Clifford algebra.
\end{remark}
\subsection{The finite Heisenberg group(s)}
\begin{remark}
While a Heisenberg group of order $n^3$ can be defined for every $n \in \mathbb{N}$, the discussion here will only hold for $n=p$ prime. This is mainly to make things easier regarding the representation theory of these groups.
\end{remark}
\begin{definition}
The Heisenberg group of order $p^3$ is the finite group given by the generators and relations
\begin{displaymath}
H_p = \langle e_1,e_2,z | e_1^p = e_2^p=z^p,[e_1,e_2] = z, e_1z = z e_1,e_2z = z e_2\rangle 
\end{displaymath}
and it is a central extension of the group $\mathbb{F}_p \times \mathbb{F}_p$
\begin{displaymath}
 \xymatrixcolsep{4pc}\xymatrix{1 \ar[r]& \mathbb{F}_p \ar[r]^-{1 \mapsto z} & H_p \ar[r]^-{e_1 \mapsto (1,0)}_-{e_2 \mapsto (0,1)}&\mathbb{F}_p \times \mathbb{F}_p \ar[r]&  1}.
\end{displaymath}
\end{definition}
All the 1-dimensional simple representations of $H_p$ are induced by the characters of $\mathbb{F}_p \times \mathbb{F}_p$. The other simple representations are $p$-dimensional and are determined by a primitive $p$th root of unity. They are defined in the following way: choose a primitive $p$th root of unity $\omega$, then define the following action of $H_p$ on the vector space $V = \mathbb{C}x_0 + \ldots + \mathbb{C}x_{p-1}$
\begin{displaymath}
e_1 \cdot x_i = x_{i-1}, e_2 \cdot x_i= \omega^i x_i, 
\end{displaymath}
indices taken $\bmod p$. Taking another primitive root gives you another simple representation. This means that there are $p^2$ 1-dimensional and $p-1$ $p$-dimensional  irreducible representations, which are all the simple ones. There are $p^2+p-1$ conjugacy classes, 1 for each central element and the other $p^2-1$ classes contain a unique element of the form $e_1^a e_2^b$, $a,b \in \mathbb{F}_p, (a,b) \neq (0,0)$.
\par The character of a simple $p$-dimensional representation $V$ is given by 
\begin{align*}
\chi_V(z^k) &= p \omega^k \\
\chi_V(e_1^a e_2^b) &= 0, (a,b) \neq (0,0).
\end{align*}
Such a representation $V$ also defines an antisymmetric bilinear form on the $\mathbb{F}_p$-vector space $\mathbb{F}_p \times \mathbb{F}_p$. Identifying $e_1$ and $e_2$ with their images in $\mathbb{F}_p \times \mathbb{F}_p$, we get this form by setting $\langle e_1,e_2 \rangle = \omega$ and extending it linearly to $\mathbb{F}_p \times \mathbb{F}_p$, thus
\begin{displaymath}
\langle a e_1 + b e_2,c e_1 + d e_2 \rangle = \omega^{ad-bc}.
\end{displaymath}
If we define a group morphism $\langle z \rangle \stackrel{\phi}{\longrightarrow} \mu_p$ by $\phi(z) = \omega$ (written multiplicatively in $\mu_p$), then we have a commutative diagram
\begin{displaymath}
\xymatrix{H_p \times H_p \ar[r] \ar[d]^-{[,]} &\mathbb{F}_p \times \mathbb{F}_p \ar[d]^-{\langle , \rangle} \\
\langle z \rangle \ar[r]^-\phi & \mu_p }
\end{displaymath}
Since every $p$-dimensional representation is determined by the image of $z$, every nontrivial antisymmetric bilinear form on $\mathbb{F}_p \times \mathbb{F}_p$ uniquely defines a simple representation of $H_p$. Conversely, every simple $p$-dimensional representation $V$ of $H_p$ defines a unique nontrivial antisymmetric bilinear form on $\mathbb{F}_p \times \mathbb{F}_p$ by extending linearly $\langle e_1,e_2 \rangle = \frac{\chi_V(z)}{p}$.

\subsection{The modular curve $X(n)$}
\label{sub:Mod}
As is well known, the modular group $\Gamma = PSL_2(\mathbb{Z})$ acts on the complex upper half-plane
\begin{displaymath}
\mathbb{H}=\{x+iy | y>0\}
\end{displaymath}
by M\"obius transformations. The fundamental domain of this action defines isomorphism classes of elliptic curves and its compactification, made by adding the $\Gamma$-orbit $\overline{\mathbb{Q}}=\mathbb{Q}\cup \{\infty\}$, is the Riemann sphere $S^2$. In general, one can take any other group $G$ of finite index in $\Gamma$, find its fundamental domain in $\mathbb{H}$ and check what information a point in this domain holds. The modular curve $X(n), n \in\mathbb{N}$ is made this way by taking $G = \Gamma(n)$, with
\begin{displaymath}
\Gamma(n) = \left\lbrace \begin{bmatrix}
a & b \\ c & d
\end{bmatrix} \in \Gamma | a,d \equiv 1 \bmod n , b,c \equiv 0 \bmod n \right\rbrace.
\end{displaymath}
A point on $X(n)$ holds 3 pieces of information:
\begin{itemize}
\item an elliptic curve $(E,O)$,
\item an embedding of $\mathbb{Z}/n\mathbb{Z} \times \mathbb{Z}/n\mathbb{Z}$ into $E$ or equivalently, two generators $e_1,e_2$ of $E[n] = \{ P \in E| [n]P = O\}$.
\item a primitive $n$th root of unity $\omega$ such that $\langle e_1,e_2 \rangle = \omega$, where this bilinear antisymmetric form is found by the Weil-pairing.
\end{itemize}
$X(n)$ has an action of $PSL_2(n) = \Gamma/\Gamma(n)$ by definition. This action is defined by taking another set of generators of $E[n]$, $f_1,f_2$, but their inner product must still remain $\omega$. This defines an $SL_2(\mathbb{Z}/n\mathbb{Z})$-action, but since $-I_2$ works trivially on $X(n)$, we have an $PSL_2(\mathbb{Z}/n\mathbb{Z})$ action.
\par Now let $n=p$ be prime. As we have seen in last subsection, a bilinear antisymmetric form on $\mathbb{F}_p \times \mathbb{F}_p$ defines a simple $H_p$-representation $V$.
Let $P_1,P_2$ be a generating set of $E[p]$ and denote $P_{a,b} = [a]P_1+[b]P_2$, then there exists a (unique up to multiplication with a scalar) function $f$ on $E$ with divisor
\begin{displaymath}
-(P_{0,0} + \ldots + P_{0,p-1})+ P_{p-1,0}+\ldots + P_{p-1,p-1}.
\end{displaymath}
In \cite{silverman2009arithmetic} it is proved that there exists a primitive $p$th root of unity such that $\omega = \frac{f}{\phi_{P_2}^*(f)}$, where $\phi_{P}^*$ stands for the pullback under the morphism
\begin{displaymath}
E \stackrel{\phi}{\longrightarrow} E, \tau \mapsto \tau + P.
\end{displaymath}
Calculating the divisor, one finds that the function 
\begin{displaymath}
N(f)=f \phi_{P_1}^*(f)\ldots (\phi_{P_1}^*)^{p-1}(f)
\end{displaymath}
is constant and not 0, which means we can rescale $f$ so that $N(f) = 1$. We will now define an action of $H_p$ on the vector space
\begin{displaymath}
\mathcal{L}(P_{0,0} + \ldots + P_{0,p-1}) = H^0(E,\mathcal{O}(P_{0,0} + \ldots + P_{0,p-1})).
\end{displaymath}
Let $x_0 = 1$ and define
\begin{align*}
e_1 \cdot g = f\phi_{P_1}^*(g),\\ e_2 \cdot g= \phi_{P_2}^*(g).
\end{align*}
If we set $x_i =e_1^{p-i}\cdot x_0$, we find that 
\begin{align*}
e_1\cdot x_i = x_{i-1}, \\ e_2\cdot x_i = \omega^i x_i.
\end{align*}
This defines our action of $H_p$. These global sections define an embedding of $E$ into $\mathbb{P}^{p-1}$ and it is clear that the defining equations will be $H_p$-invariant.
\section{The case $n=3$}
Before we start with the case $n=5$, let us see what happens when $n = 3$. In this case, the classification of Artin-Schelter regular algebras has been done, see \citep{artin1987graded} and \citep{artin2007some}.
\par When $n=3$, the algebras we want to study have the following relations for $t \in \mathbb{C}$
\begin{align*}
xy+yx=t z^2, && yz+zy=t x^2, && zx+xz=t y^2.
\end{align*}
\begin{theorem}
For generic values of $t$, the $H_3$-Clifford algebra $\mathfrak{C}(t)$ is a Sklyanin algebra associated to the elliptic curve
\begin{displaymath}
E \leftrightarrow t(x^3+y^3+z^3)+(2-t^3)xyz, O = (1,-1,0)
\end{displaymath}
and translation by the point $(1:1:-t)$, which is a point of order 2. 
\end{theorem}
There are however $7$ values for $t \in \mathbb{C}$ where the corresponding $H_3$-Clifford algebra is not a Sklyanin algebra: $t=2,2\omega,2\omega^2,-1,-\omega,-\omega^2,0$, with $\omega$ being a primitive third root of unity. The noncommutative algebra corresponding to $t=\infty$ is the algebra with relations
\begin{align*}
z^2=x^2=y^2=0
\end{align*}
and is clearly not regular, so we have a total of 8 points in $\mathbb{P}^1$ where we don't have a Sklyanin algebra. When $t = 0,2,2\omega,2\omega^2$, the corresponding algebra is still regular, but in the other $4$ cases the Koszul dual $\mathfrak{C}(t)^!$ does not define the empty set. For $t = -1,-\omega,-\omega^2, \infty$, $\mathfrak{C}(t)^!$ defines a set of three points
\begin{flalign*}
t=-1 \longleftrightarrow& \{(1:1:1),(1:\omega:\omega^2),(1:\omega^2:\omega)\}, \\
t=-\omega \longleftrightarrow& \{(1:1:\omega^2),(1:\omega:\omega),(1:\omega^2:1)\}, \\
t=-\omega^2 \longleftrightarrow &\{(1:1:\omega),(1:\omega:1),(1:\omega^2:\omega^2)\},\\
t=\infty \longleftrightarrow &\{(1:0:0),(0:1:0),(0:0:1)\}.
\end{flalign*}
Notice that we have an action of $PSL_2(3)$ on the moduli space $\mathbb{P}^1$ such that algebras in the same orbit are isomorphic, but the $H_3$-representation is twisted (but still as an $H_3$-module isomorphic to $V_1$): suppose that we have an algebra with relations
\begin{align*}
xy+yx=t z^2, && yz+zy=t x^2, && zx+xz=t y^2.
\end{align*}
Then we can twist the Heisenberg action by an automorphism $\phi$ of $H_3$ and take the eigenvector with eigenvalue $1$ of $\phi(e_2)$ and its orbit under $\phi(e_1)$ as generators of this algebra (just the same as when $\phi=Id$, where $x_0$ is an eigenvector with eigenvalue 1 of $e_2$ and $x_i = e_1^{-i} x_0$). We want the twist to preserve the character of the representation, which means that we need to preserve the antisymmetric inner product on $\mathbb{F}_3 \times \mathbb{F}_3$ defined by $\langle e_1,e_2 \rangle = \omega$ (written multiplicatively), where $e_1$ and $e_2$ are identified with their images in $\mathbb{F}_3\times \mathbb{F}_3$ like in the previous section.
\par The condition $\langle \phi(e_1), \phi(e_2) \rangle = \omega$ means that we need to look at elements of $SL_2(3)$. The element $-I_2$ however will work trivially on $\mathbb{P}^1$, so it will be a $PSL_2(3)$-action. Using the following generators
\begin{displaymath}
U=
\begin{bmatrix}
0 & -1 \\ 1 & 0
\end{bmatrix},
 V=
\begin{bmatrix}
0 & 1 \\ -1 & 1
\end{bmatrix},
\end{displaymath}
with $U^2 =V^3=I$, the action of $PSL_2(3)\cong A_4$ is given by the following M\"obius transformations:
 \begin{displaymath}
U\leftrightarrow U'(t) =\frac{-t+2}{t+1},
 V\leftrightarrow V'(t)=\frac{-\omega^2 t+2}{\omega^2 t+1}.
\end{displaymath}
This action is a \emph{right} action, for this action comes naturally by twisting the group morphism $\varphi: H_3 \rightarrow Aut(\mathfrak{C}(t))$ by an automorphism of $H_3$. So suppose that we have an automorphism $\psi:H_3 \rightarrow H_3$, then we get a new group morphism given by $\varphi \circ \psi: H_3 \rightarrow Aut(\mathfrak{C}(t))$.
\par Using this action, the sets $\{0,2,2\omega,2\omega^2\}$ and $\{\infty, -1,-\omega,-\omega^2\}$ correspond to the action of $PSL_2(3)$ on $\mathbb{P}^1_{\mathbb{F}_3}$.
\begin{example}
Suppose we want to calculate the action of 
\begin{displaymath}
M=VU=\begin{bmatrix}
1 & 0 \\ 1 & 1
\end{bmatrix}
\end{displaymath}
on our moduli space $\mathbb{P}^1$. Taking our algebra $\mathfrak{C}(t)$ (and we allow $t = \infty$), this means we need to find the eigenvector of $M \cdot e_2$ with eigenvalue $1$. In this particular case, this eigenvector is still $x_0$, but now $x_2$ is defined as $(e_1 e_2)x_0$ instead of $e_1 x_0$. So if we take $y_0,y_1,y_2$ as the generators of our isomorphic algebra, we have
\begin{align*}
y_0 &= x_0,\\
y_1 &= \omega^2 x_1, \\
y_2 &= x_2.
\end{align*}
A quick calculation shows that $y_0 y_1 + y_1 y_0 = \omega^2 t y^2_2$ and so the corresponding M\"obius transformation is given by $t \mapsto \omega^2 t$, with fixed points $0$ and $\infty$. This shows it is a right action: the composition $U'\circ V'$ is equal to the action of $M$, but $M = VU$.
\end{example}

In this case, the action of $PSL_2(3)$ is not a projectification of a 2-dimensional representation, because all $PSL_2(3)$-representations of dimension 2 would be direct sums of 1-dimensional representations. This is impossible because then every element of order 2 would act trivially and this is clearly not the case.
\par One other thing of importance is that we have a duality between the quantum spaces (the algebras isomorphic to the algebra $\mathfrak{C}(0)$) and the 4 nonregular algebras, given by the point varieties of the regular ones and the Koszul dual of the nonregular ones. Before we give this duality, we recall the definition of a point module.
\begin{definition}
Let $B$ be a positively graded, connected $\mathbb{C}$-algebra generated in degree 1. A point module $M$ is a cyclic, graded left $B$-module with Hilbert series $\frac{1}{1-t}$ with $M = B M_0$.
\end{definition}
\begin{example}
For the quantum space $\mathbb{C}_{-1}[x_0,\ldots,x_{p-1}]$, the point modules are described by the full graph on $p$ points, the $p$ vertices given by the $H_p$-orbit of the point $(1:0:\ldots:0)$.
\end{example}
\par The duality we want to show is given in the following way: let $\mathfrak{C}(t_1)$ be an algebra isomorphic to the quantum plane. Then its point variety (the point modules of this algebra) is given by the full graph on 3 points, the 3 points corresponding to a unique $H_3$-orbit. Then there exists a unique nonregular algebra $\mathfrak{C}(t_2)$ in our moduli space such that its Koszul dual $\mathfrak{C}(t_2)^!$ has as its point variety this unique $H_3$-orbit. For example, the point variety of the quantum plane $\mathfrak{C}(0)$
\begin{displaymath}
xy+yx=xz+zx=yz+zy=0
\end{displaymath}
is given by the three lines $x=0,y=0,z=0$, and they intersect 2 by 2 in 3 points. These points are the solutions of the equations $xy=xz=yz=0$. These equations are the relations of the Koszul dual of the algebra with relations
\begin{displaymath}
x^2=y^2=z^2=0,
\end{displaymath}
which is the algebra $\mathfrak{C}(\infty)$.
\begin{theorem}
There is a 1-to-1 correspondence between the $H_3$-Clifford algebras isomorphic to $\mathbb{C}_{-1}[x_0,x_1,x_2]$ and the nonregular $H_3$-Clifford algebras, with the correspondence being given in the following way: the regular $H_3$-Clifford algebra $\mathfrak{C}(t)$ corresponds to the nonregular $\mathfrak{C}(t')$ iff $t$ and $t'$ are fixed by the same subgroup of order 3 in $PSL_2(3)$. This correspondence gives a natural bijection between $\mathbb{P}^1_{\mathbb{F}_3}$ and the points in $\mathbb{P}^1$ with corresponding $H_3$-Clifford algebra isomorphic to the quantum plane.
\end{theorem}
\begin{remark}
Although $4$ is not prime, one can ask the natural question what happens if $n = 4$. This case is however not very interesting: there aren't any quadratic algebras $\mathfrak{C}$ with defining relations $x_ix_j+x_jx_i=a_{ij}x_k^2$ and with degree 1 part isomorphic to $V_1$, $V_1$ being the $H_4$ representation with $z$ working as $iI_4$ (apart from the quantum space with $a=0$). There aren't any squares of elements in $V_1$ on which $e_2$ works as multiplying by $\pm i$, so this forces that $x_i x_j + x_j x_i = 0$ whenever $j-i \equiv 1 \bmod 2$. When $j-i \equiv 0 \bmod 2$, we need to define $x_0 x_2  +x_2 x_0$ and $x_1 x_3+x_3 x_1$, so we need to find elements $v$ and $w$ in degree 1 such that the action on $v^2$ and $w^2$ by $e_2$ is respectively given by multiplication with $-1$ and $1$, with the extra condition that $e_1$ permutes $v$ and $w$. Such elements are impossible to find (except when $u=v=0$), so we are done.
\end{remark}
By the last remark we have proved
\begin{theorem}
The only $H_4$-Clifford algebra is $\mathbb{C}_{-1}[x_0,x_1,x_2,x_3]$, which is regular.
\end{theorem}
\section{The case $n=5$}
From now on, $\omega = e^\frac{2\pi i}{5}$. Similar to the case of $n=3$, we have a right action of $PSL_2(5) \cong A_5$ on $\mathbb{P}^2$ that gives isomorphic algebras in the moduli space. This action is found using the exact same procedure as in the previous section. Again using $U,V$ as generators of $PSL_2(5)$, we get the following matrices
\begin{displaymath}
U\leftrightarrow
\frac{1}{\sqrt{5}}\begin{bmatrix}
\omega^2+\omega^3 & \omega+\omega^4 & 2 \\ \omega+\omega^4 & \omega^2+\omega^3 & 2
\\ 1&1&1
\end{bmatrix},
V\leftrightarrow
 \frac{1}{\sqrt{5}}\begin{bmatrix}
\omega+\omega^2 & \omega^2+1 & 2 \\ \omega^3+1 & \omega^3+\omega^4 & 2 \\ \omega^4 & \omega &1
\end{bmatrix}
\end{displaymath}
and now this projective representation comes from a simple representation of $A_5$, the icosahedron representation. The factor $\frac{1}{\sqrt{5}}$ is necessary to get an $A_5$ representation, but it doesn't matter for our algebras. To discuss these algebras, we will sometimes use the algebras which correspond to points at infinity, because their equations are easier to work with.
\subsection{The nonregular algebras}
The nicest algebras are those that are noetherian domains, so to find them we will use theorem \ref{th:Reg}. Using this theorem and the fact that the equations of the quadric system correspond to the relations of the Koszul dual, we need to analyse the following equations
\begin{subequations}
\begin{equation}
x_0^2+ax_1x_4+bx_2x_3=0,
\end{equation}
\begin{equation}
x_1^2+ax_2x_0+bx_3x_4=0,
\end{equation}
\begin{equation}
x_2^2+ax_3x_1+bx_4x_0=0,
\end{equation}
\begin{equation}
x_3^2+ax_4x_2+bx_0x_1=0,
\end{equation}
\begin{equation}
x_4^2+ax_0x_3+bx_1x_2=0
\end{equation}
\label{eq:koszul}
\end{subequations}
and determine when the only solution is given by $(0,0,0,0,0)$.
\par We can also look at the $H_5$-Clifford algebras corresponding to points on the line at infinity, so that we get a projective variety for every point in our moduli space $\mathbb{P}^2$. We do this by changing $(a,b)$ to $(A:B:C)$ and putting a $C$ before every $x_i^2$. However, since every point on the line at infinity of $\mathbb{P}^2$ is in the $PSL_2(5)$-orbit of an affine point, this is not always interesting. The points at infinity however will give equations that are easier to handle in some cases.
\begin{theorem}
Generically, the $H_5$-Clifford algebra $\mathfrak{C}(A:B:C)$ is a regular graded Clifford algebra.
\end{theorem}
\begin{proof}
Calculating the Gr\"obner basis of the relations of $\mathfrak{C}(a,b)^!$ using Mathematica, one finds that the following monomial belongs to the ideal $I$ defined by the equations from \ref{eq:koszul}
\begin{equation}
(1+a^5-4ab+a^6b+5a^3b^3+b^5+ab^6)x_4^6.
\end{equation}
This means that, if $1+a^5-4ab+a^6b+5a^3b^3+b^5+ab^6 \neq 0$, $x_4^6 \in I$. But $I$ is closed under the action of the Heisenberg group, so this implies that every $x_i^6$ belongs to $I$. This implies that $I$ defines the empty set and so $\mathfrak{C}(a,b)$ would have all the good properties we desire.
\end{proof}
\par So generically, we get a regular graded Clifford algebra. The only possible `bad' algebras are given by the curve $1+a^5-4ab+a^6b+5a^3b^3+b^5+ab^6 =0$. Adding the line $C=0$, decomposing the equation  $1+a^5-4ab+a^6b+5a^3b^3+b^5+ab^6 =0$ and looking at the projective closure of this variety, we get 6 lines and a conic section
\begin{subequations}
\begin{equation}
C = 0,
\end{equation}
\begin{equation}C+A+B=0,
\end{equation} 
\begin{equation}
C+\omega A + \omega^4 B=0,
\end{equation}
\begin{equation}
C+\omega^4A+\omega B=0,
\end{equation} 
\begin{equation}
C+\omega^2 A + \omega^3 B=0
\end{equation}
\begin{equation}
C+\omega^3 A + \omega^2 B=0,
\end{equation}
\begin{equation}
AB + C^2 = 0,
\end{equation}
\label{eq:6lines}
\end{subequations}
where the regularity condition possibly fails. Notice that the 6 lines we found form an orbit under $PSL_2(5)$ (as we expected). This means that, if we want to analyse how `far' these algebras are from being Artin-Schelter regular, we can restrict ourselves to studying the points on the line $C=0$ and on the conic section $AB+C^2$. For the line $C=0$, this boils down to describing the relations of the form
\begin{subequations}
\begin{equation}
Ax_1x_4+Bx_2x_3=0,
\end{equation}
\begin{equation}
Ax_2x_0+Bx_3x_4=0,
\end{equation}
\begin{equation}
Ax_3x_1+Bx_4x_0=0,
\end{equation}
\begin{equation}
Ax_4x_2+Bx_0x_1=0,
\end{equation}
\begin{equation}
Ax_0x_3+Bx_1x_2=0,
\end{equation}
\label{eq:lininf}
\end{subequations}
excluding the case $(A,B)=(0,0)$. For a generic point on the line $C=0$, one can use Macaulay2 and find that the Hilbert series is given by
\begin{displaymath}
\frac{1+4t+5t^2-5t^4}{1-t}.
\end{displaymath}
So it follows that generically, the $H_5$-orbit of $(1:0:0:0:0)$ is the solution set of the equations from \ref{eq:lininf}. However, when you calculate the Gr\"obner basis (which determines the Hilbert series) using Mathematica, there are 7 points on this line where this (possibly) fails:
\begin{align*}
(A:B:C)&=(1:-\omega^k:0), k = 0\ldots 4,\\
(A:B:C)&=(1:0:0),\\
(A:B:C)&=(0:1:0).
\end{align*} 
The Hilbert series indeed fails here. Calculating Hilbert series, we find
\begin{theorem}
For a point on the line $C = 0$, the algebra $\mathfrak{C}(A:B:C)^!$ determines the $H_5$-orbit of $(1:0:0:0:0)$ as Zariski-closed subset of $\mathbb{P}^4$, except in the following 7 cases:
\begin{itemize}
\item For $k=0,\ldots,4$, the point $(1:-\omega^k:0)$ is the intersection of the line $C=0$ and the line $C + \omega^{3k} A +\omega^{-3k} B = 0$. At the point $(1:-\omega^k:0)$, the projective variety determined by $\mathfrak{C}(1:-\omega^k:0)^!$ is given by 10 points. These 10 points form the union of 2 $H_5$-orbits, each orbit consisting of $5$ elements. Representatives of these 2 orbits are given by $(1:0:0:0:0)$ and $(1:1:\omega^{-2k}:\omega^{-k}:\omega^{-2k})$. The Hilbert series of the commutative algebra $\mathfrak{C}(A:B:C)^!$ is given by
\begin{displaymath}
\frac{1+4t+5t^2}{1-t}.
\end{displaymath} 
\item For $(A:B:C)=(1:0:0)$, we get 5 lines, $\cup_{i=0}^4 \mathbf{V}(x_i,x_{i+1},x_{i+2})$. The Hilbert series $\mathfrak{C}(1:0:0)^!$ is given by
\begin{displaymath}
\frac{1+3t+t^2}{(1-t)^2}.
\end{displaymath}
The corresponding configuration with its vertices the 5 points in the orbit of $(1:0:0:0:0)$ is given by figure \ref{fig:Configuration 1}.
\begin{center}
\begin{figure}[H]
\begin{tikzpicture}[style=thick]
\draw (18:3cm) circle (2pt)  node[above right=-1.75pt]{$(0:1:0:0:0)$} -- (90:3cm);
\draw (90:3cm) circle (2pt)  node[above]{$(1:0:0:0:0)$} -- (90+72:3cm);
\draw (90+72:3cm) circle (2pt)  node[above left=-1.75pt]{$(0:0:0:0:1)$} -- (90+72+72:3cm);
\draw (90+72+72:3cm) circle (2pt)  node[below]{$(0:0:0:1:0)$} -- (90+72+72+72:3cm);
\draw (90+72+72+72:3cm) circle (2pt)  node[below]{$(0:0:1:0:0)$} -- (90+72+72+72+72:3cm);
\end{tikzpicture}
\caption{First configuration}
\label{fig:Configuration 1}
\end{figure}
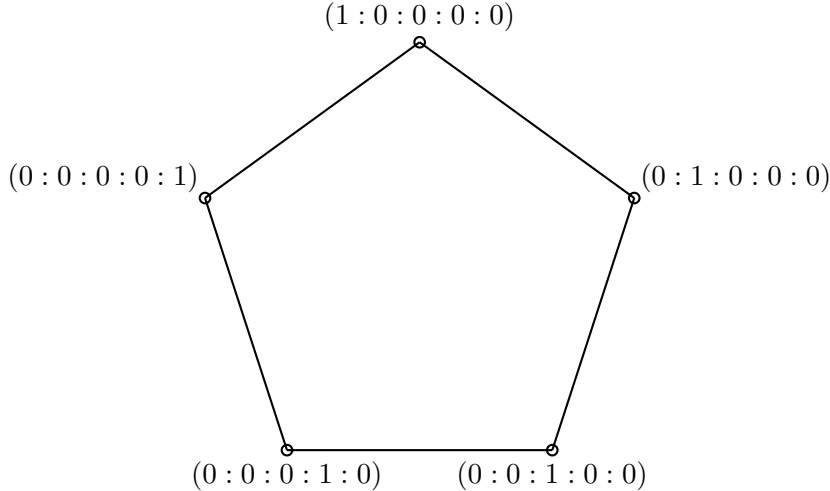
\end{center}
\item For $(A:B:C) = (0:1:0)$, there are again 5 lines, now given by $\cup_{i=0}^4 \mathbf{V}(x_i,x_{i+1},x_{i+3})$. The Hilbert series of $\mathfrak{C}(0:1:0)^!$ is given by
\begin{displaymath}
\frac{1+3t+t^2}{(1-t)^2}.
\end{displaymath}
The corresponding configuration is given by figure \ref{fig:Configuration 2}.
\begin{center}
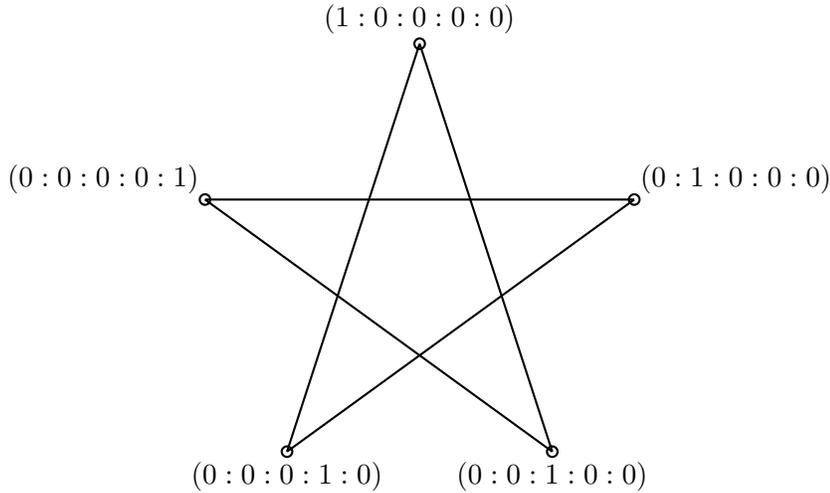
\begin{figure}[H]
\begin{tikzpicture}[style=thick]
\draw (18:3cm) circle (2pt)  node[above right=-1.75pt]{$(0:1:0:0:0)$} -- (90+72:3cm);
\draw (90+72:3cm) circle (2pt)  node[above left=-1.75pt]{$(0:0:0:0:1)$} -- (90+72+72+72:3cm);
\draw (90+72+72+72:3cm) circle (2pt)  node[below]{$(0:0:1:0:0)$} -- (90:3cm);
\draw (90:3cm) circle (2pt)  node[above]{$(1:0:0:0:0)$} -- (90+72+72:3cm);
\draw (90+72+72:3cm) circle (2pt)  node[below]{$(0:0:0:1:0)$} -- (18:3cm);
\end{tikzpicture}
\caption{Second configuration}
\label{fig:Configuration 2}
\end{figure}
\end{center}
\end{itemize}
\end{theorem}
For a point on the conic section $AB+C^2$, the relations from equation \ref{eq:koszul} determine a smooth genus 1 curve, except when $A=0$, $B=0$ or 
\begin{align*}
(A:B:C) &= (\omega^k (\omega^2+\omega^3),\omega^{-k}(\omega+\omega^4),1),k=0,\ldots 4,\\
(A:B:C) &= (\omega^k (\omega+\omega^4),\omega^{-k}(\omega^2+\omega^3),1),k=0,\ldots 4.
\end{align*}
These 12 points are of course the intersection points with the 6 lines from equation \ref{eq:6lines}. For every affine point $(a,\frac{-1}{a})$ excluding the 12 special points, the point $(0:1:a:-a:-1)$ and its $H_5$-orbit lies on the curve defined by \ref{eq:koszul}. Taking this point to be the point $O$, the corresponding curve is an elliptic curve $E$, and the $H_5$-orbit of $O$ gives an embedding of $\mathbb{F}_5 \times \mathbb{F}_5$ in $E$. It follows from calculations using Mathematica and Macaulay2 that every point on $\mathbf{V}(AB+C^2)$ determines an elliptic curve, except for 12 points, which are the intersection points of the 6 lines from equation \ref{eq:6lines}. A $PSL_2(5)$-orbit on this conic sections determines an isomorphism class of elliptic curves, but changes the chosen generators of $E[5]$. These new generators however have the same Weil pairing as the original generators (as they determine the representation in degree 1, which is unchanged) and therefore the conic section $\mathbf{V}(AB+C^2)-\{12 \text{ points}\}$ with the $PSL_2(5)$-action is a model for the modular curve $X(5)$. The extra 12 points give the compactification of $X(5)$. Summarizing, we have the following theorem:
\begin{theorem}
For the points on the curve $\mathbf{V}(AB+C^2)-\{12 \text{ points}\}$, the corresponding algebra $\mathfrak{C}(A:B:C)^!$ is the homogeneous coordinate ring of an elliptic curve $E$ embedded in $\mathbb{P}^4$ with $O=(0:1:a:-a:-1)$. Furthermore, the point $(A:B:C)$ determines an embedding of $\mathbb{F}_5 \times \mathbb{F}_5$ in $E$ and a fixed $\mathbb{F}_5$-basis $(e_1 \cdot O, e_2 \cdot O)$. Conversely, every point in $X(5)$ determines a unique point on $\mathbf{V}(AB+C^2)$ and thus $\mathbf{V}(AB+C^2)$ is a model for $X(5)$, with the 12 extra points the points needed to make the compactification $\overline{X(5)}$.
\end{theorem}
\par Now, choose a line from equations \ref{eq:6lines}. This line will intersect $\mathbf{V}(AB+C^2)$ in exactly 2 points, $P_1$ and $P_2$. The commutative algebra corresponding to $P_1$ will describe the union of 5 lines that intersect 2 by 2. The intersection points will be a $H_5$-orbit of 5 elements and the configuration will be like figure \ref{fig:Configuration 1}. The point $P_2$ will also determine the union of 5 lines that intersect 2 by 2, with the intersection points the same as for $P_1$, but now the configuration will be like figure \ref{fig:Configuration 2}.
\subsection{Being Koszul}
This section gives a proof of the following theorem.
\begin{theorem}
Every $H_5$-Clifford algebra $\mathfrak{C}(A:B:C)$ is Koszul, except for the points on the 6 lines given by the $PSL_2(5)$-orbit of $C=0$ that do not lie on the conic section $\mathbf{V}(AB+C^2)$.
\end{theorem}
\begin{proof}
When $\mathfrak{C}(A:B:C)$ is regular (the generic case), we can use theorem \ref{th:Reg} and apply theorem 2.2 of \cite{shelton2001koszul} to conclude that  $\mathfrak{C}(A:B:C)$ is regular. When the Koszul dual is an elliptic curve, it follows from the Koszulity of the homogeneous coordinate ring of an elliptic curve that the algebras of the form (choosing representatives such that the sum of the indices is 0)
\begin{align*}
x_1x_4+x_4x_1 = a x_0^2, && x_1x_4+x_4x_1 = \frac{-1}{a} x_0^2
\end{align*}
are indeed Koszul whenever the point $(a,\frac{-1}{a})$ determines an elliptic curve. For the 12 points on $\mathbf{V}(AB+C^2)$ that determine 5 lines, Koszulity follows from the fact that in this case $\mathfrak{C}(A:B:C)$ is isomorphic to $\mathfrak{C}(1:0:0)$. The relations of $\mathfrak{C}(1:0:0)^!$ are given by 5 monomials of degree 2 (and the commutator relations of course). Then theorem 3.15 of \citep{conca2013koszul} gives us that $\mathfrak{C}(1:0:0)^!$ is indeed Koszul, but then $\mathfrak{C}(1:0:0)$ is also Koszul.
\par These are all the Koszul algebras. Assume that $(A:B:C)$ lies on 1 of the 6 lines from \ref{eq:koszul} (excluding the 12 points lying on the conic section and the 15 points at the intersections) and that $\mathfrak{C}(A:B:C)$ is Koszul. Then $\mathfrak{C}(A:B:C)$ has as Hilbert series
\begin{displaymath}
\frac{1+t}{1-4t+5t^2-5t^4}= 1+5 t+15 t^2+35 t^3+70 t^4+130 t^5+\ldots
\end{displaymath}
This would mean that the Hilbert series of this algebra is equal to $\frac{1}{(1-t)^5}$ up to degree 4. But then it follows from the fact that this algebra has as a quoti\"ent a graded Clifford algebra that it should have Hilbert series $\frac{1}{(1-t)^5}$, which is not the case.
\par The same argument also applies when the Koszul dual of $\mathfrak{C}(A:B:C)$ determines 10 points, because now its Hilbert series should be
\begin{displaymath}
\frac{1+t}{1-4t+5t^2} = 1+5 t+15 t^2+35 t^3+65 t^4+85 t^5+\ldots
\end{displaymath}
and this is clearly not $\frac{1}{(1-t)^5}$.
\end{proof}
\subsection{The quantum planes and their relationship with the 6 lines}
$PSL_2(5)$ has 6 5-Sylow groups, each one stabilizing one of the 6 lines of equation \ref{eq:6lines}. However, such a group doesn't fix the entire line, but rather 2 points (the intersection with $\mathbf{V}(AB+C^2)$). Since these groups are cyclic and they are all conjugated, it suffices to work with one 5-Sylow group, e.g. the group generated by $VU$. The corresponding matrix is given by
\begin{displaymath}
\begin{bmatrix}
1 & 0 \\
1 & 1 
\end{bmatrix}
\leftrightarrow\begin{bmatrix}
\omega^4 & 0 & 0 \\ 0 & \omega & 0
\\ 0&0&1
\end{bmatrix},
\end{displaymath}
so the third point fixed by this matrix is $(0:0:1)$, with corresponding algebra the quantum space $\mathfrak{C}(0:0:1)$ with relations $x_k x_l +x_l x_k = 0, k \neq l$. The other points in the $PSL_2(5)$-orbit of $(0:0:1)$ are given by $(2\omega^k:2\omega^{-k}:1), k = 0,\ldots 4$. This gives us 6 points with corresponding Clifford algebra isomorphic to the quantum space $\mathfrak{C}(0:0:1)$, with the $PSL_2(5)$-action the same as its action on $\mathbb{P}^2_{\mathbb{F}_5}$. 
\par Again we have a certain duality regarding these quantum algebras and the 6 lines from equation \ref{eq:6lines}. I will explain it using the example of $(0:0:1)$ and the line $C=0$. The point modules of the quantum plane in this case are parametrized by the full graph on 5 points, with its vertices given by the $H_5$-orbit of $(1:0:0:0:0)$. The configuration is given by figure \ref{fig:fullgraph}.
\begin{center}
\begin{figure}[H]
\begin{tikzpicture}[style=thick]
\draw (18:3cm) circle (2pt)  node[above right=-1.75pt]{$(0:1:0:0:0)$} -- (90:3cm);
\draw (18:3cm) circle (2pt) -- (90+72:3cm);
\draw (18:3cm) circle (2pt) -- (90+2*72:3cm);
\draw (90:3cm) circle (2pt) -- (90+3*72:3cm);
\draw (90:3cm) circle (2pt) -- (90+2*72:3cm);
\draw (90+72:3cm) circle (2pt) -- (90+3*72:3cm);
\draw (90:3cm) circle (2pt)  node[above]{$(1:0:0:0:0)$} -- (90+72:3cm);
\draw (90+72:3cm) circle (2pt)  node[above left=-1.75pt]{$(0:0:0:0:1)$} -- (90+72+72:3cm);
\draw (90+72+72:3cm) circle (2pt)  node[below]{$(0:0:0:1:0)$} -- (90+72+72+72:3cm);
\draw (90+72+72+72:3cm) circle (2pt)  node[below]{$(0:0:1:0:0)$} -- (90+72+72+72+72:3cm);
\end{tikzpicture}
\caption{The full graph on 5 points}
\label{fig:fullgraph}
\end{figure}
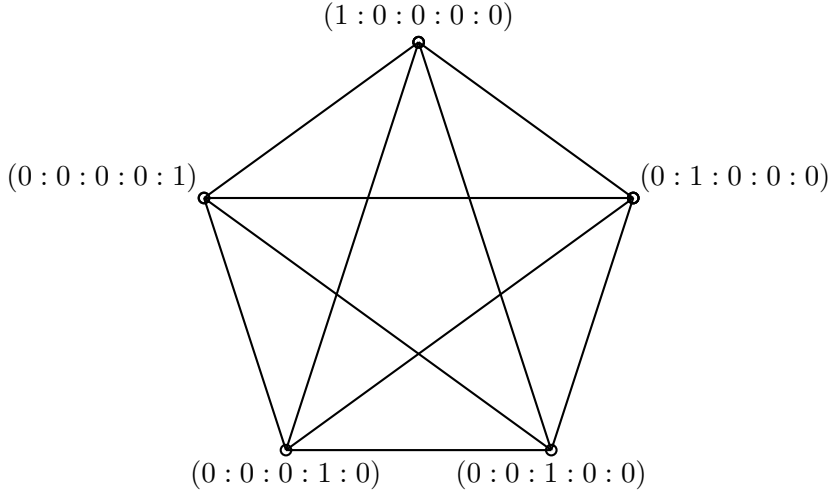
\end{center}
The graph of figure \ref{fig:fullgraph} is the union of figure \ref{fig:Configuration 1} and \ref{fig:Configuration 2}, which were projective varieties determined by the $\mathfrak{C}(1:0:0)^!$ and $\mathfrak{C}(0:1:0)^!$. More generally, every point on the line $C=0$ has the $H_5$-orbit of $(1:0:0:0:0)$ in the zeroset of $\mathfrak{C}(A:B:C)^!$, which are the vertices of the graph of figure \ref{fig:fullgraph}. Summarizing, we have
\begin{theorem}
There is a 1-to-1 correspondence between the $PSL_2(5)$-orbit of $(0:0:1)$ and the $PSL_2(5)$-orbit of the line $C=0$. Moreover, for every line in the $PSL_2(5)$-orbit of the line $C=0$, there are 2 points, the intersections with $\mathbf{V}(AB+C^2)$, whose Koszul dual determines a graph on 5 points. The union of these 2 graphs gives the full graph on 5 points, which is the point variety of the corresponding algebra isomorphic to the quantum space.\\
The 2 points of $\mathbf{V}(AB+C^2)$ that correspond to the same quantum space are determined by the fact that they are fixed by the same cyclic subgroup of order $5$ in $PSL_2(5)$.
\end{theorem}
\section{Generalities for prime dimension}
In this section, $\omega = e^{\frac{2\pi i}{p}}$ with $p\geq 5$ prime. Many things described in the previous section for the case $p=5$ can be arbitrarily extended to every dimension $p$, $p$ prime. The relations we are now interested in are given by the $H_p$-action on (indices are taken $\bmod p$)
\begin{align*}
a_0(x_0x_i + x_i x_0)      = a_i x_{\frac{i}{2}}^2, 0 < i \leq \frac{p-1}{2},
\end{align*}
so our moduli space is given by $\mathbb{P}^{\frac{p-1}{2}}$ and again we have a $PSL_2(p)$ action on this space.
\subsection{The quantum spaces}
We will prove here the following theorem.
\begin{theorem}
There are exactly $p+1$ points in $\mathbb{P}^{\frac{p-1}{2}}$ for which the corresponding algebra $\mathfrak{C}(a_0:\ldots:a_{\frac{p-1}{2}})$ is isomorphic to the algebra with relations $x_i x_j + x_jx_i = 0,i \neq j$ and they form an orbit under the $PSL_2(p)$-action.
\label{th:quantum}
\end{theorem}
Before we prove this, some considerations must be made. First of all, the elements of $\mathbb{P}^{p-1}$ with a nontrivial stabilizer in $H_p$ must be found. This is equivalent to finding the eigenvectors of all the elements of $H_p$. Since a central element has every vector as eigenvector (and works therefore as the identity on $\mathbb{P}^{p-1}$), it is sufficient to consider elements of the form $e_1^k e_2^l$. Since $e_1 e_2 = z e_2 e_1$ with $z$ central, we can take powers of $(e_1^k e_2^l)^m$ until we get something of the form $z^r e_1^{-1} e_2^t$, unless $k=0$. Disregarding the part $z^r$, it suffices to determine the eigenvectors for the elements $e_1^{-1} e_2^t,t=0,\ldots, p-1$ and the element $e_2$. Of course, points belonging to the same orbit have the same stabilizer (since the $H_p$-action on $\mathbb{P}^{p-1}$ is actually an action of $\mathbb{F}_p \times \mathbb{F}_p$ and this group is commutative), so it is sufficient to give one representative of such an orbit. This means we have a total of $p+1$ orbits in $\mathbb{P}^{p-1}$ with the property that such an orbit consists of $p$ elements instead of $p^2$ elements. 
\begin{lemma}
The fixed points of $e_1^{-1}e_2^k, k=0,\ldots,p-1$ are given by the $H_p$-orbit of the following element:
\begin{displaymath}
(1:\omega^k:\omega^{3k}:\ldots : \omega^{k\frac{p(p-1)}{2}}),
\end{displaymath}
which has as its $i$th co\"ordinate $\omega^{k\frac{i(i+1)}{2}}$, if we start counting from 0.
The fixed points of $e_2$ are given by the $H_p$-orbit of $(1:0:\ldots:0)$. These are all the points with a nontrivial stabilizer.
\label{lem:fixpoint}
\end{lemma}
\begin{proof}
We again identify the image of $e_1$ and $e_2$ in $\mathbb{F}_p \times \mathbb{F}_p$ with the generators of $H_p$. Since the order of $\mathbb{F}_p \times \mathbb{F}_p$ is $p^2$, every point in $\mathbb{P}^{p-1}$ has as its stabilizer the entire group, nothing, or a cyclic group of order $p$. The entire group is impossible, since the action of $H_p$ on $\mathbb{C}^p$ was irreducible. Every cyclic group of order $p$ in $\mathbb{F}_p \times \mathbb{F}_p$ has a generator of the form $e_2$ or $e_1^{-1}e_2^k, k=0,\ldots,p-1$. Since the eigenvalues of the matrices of the representation are all distinct, there are exactly $p$ points that are fixed by any such element. So if we check that the element claimed by the theorem indeed is fixed by $e_1^{-1}e_2^k$, we are done (the claim for the group generated by $e_2$ is trivial).
By a calculation, we get
\begin{align*}
&e_1^{-1}e_2^k(1:\ldots:\omega^{k\frac{i(i-1)}{2}}:\omega^{k\frac{i(i+1)}{2}}:\ldots : \omega^{k\frac{p(p-1)}{2}})\\
=& e_1^{-1}(1:\ldots:\omega^{k\frac{i(i-1)}{2}+(i-1)k}:\omega^{k\frac{i(i+1)}{2}+ik}:\ldots : \omega^{k\frac{p(p-1)}{2}+k(p-1)})\\
=&(\omega^{k\frac{p(p-1)}{2}+k(p-1)}:\ldots:\underbrace{\omega^{k\frac{i(i-1)}{2}+(i-1)k}}_{i}:\underbrace{\omega^{k\frac{i(i+1)}{2}+ik}}_{i+1}:\ldots: \omega^{k\frac{(p-2)(p-1)}{2}+k(p-2)})
\end{align*}
Now, we have that $\omega^{k\frac{p(p-1)}{2}+k(p-1)} = \omega^{-k}$, so we may multiply each coordinate with $\omega^k$ and this gives us
\begin{align*}
=&(1:\ldots:\underbrace{\omega^{k\frac{i(i+1)}{2}}}_{i}:\underbrace{\omega^{k\frac{(i+2)(i+1)}{2}}}_{i+1}:\ldots: \omega^{k\frac{p(p-1)}{2}}).
\end{align*}
\end{proof}
Calculating what the action of $U$ on $(1:0:\ldots:0) \in \mathbb{P}^{\frac{p-1}{2}}$ in the moduli space exactly is, we see that $U$ switches the point $(1:0:\ldots:0)$ with $(1:2:\ldots:2)$. Next, the action of $VU$ fixes the point  $(1:0:\ldots:0)$, but doesn't fix the point $(1:2:\ldots:2)$. Since the order of $VU$ is prime, this means that the $VU$-orbit of $(1:2:\ldots:2)$ consists of $p$ elements. Combining these 2 observations, we have
\begin{lemma}
The $PSL_2(p)$-orbit of $(1:0:\ldots:0)$ consists of at least $p+1$ elements.
\label{lem:p1elem}
\end{lemma}
\begin{proof}\emph{[Theorem \ref{th:quantum}]}
The quantum space $\mathfrak{C}(1:0:\ldots:0:0)$ has the following property: there are exactly $p$ points in $\mathbb{P}^{p-1}$ for which the associated quadratic form has rank 1 and these $p$ points form an orbit under the Heisenberg action. So every algebra determined by a point in our moduli space $\mathbb{P}^{\frac{p-1}{2}}$ and isomorphic to the quantum space must have the same property. In \ref{lem:fixpoint} we found that there are $p+1$ different orbits consisting of $p$ points. Let $\rho=\omega^2$. We may set $a_0 = 1$, because otherwise the algebra $\mathfrak{C}(a_0:\ldots:a_{\frac{p-1}{2}})$ is not a domain and therefore it can not be isomorphic to $\mathfrak{C}(1:0:\ldots:0:0)$. So if the algebra $\mathfrak{C}(a_1,\ldots ,a_{\frac{p-1}{2}})$ is isomorphic to $\mathfrak{C}(1:0:\ldots:0:0)=\mathfrak{C}(0,\ldots,0)$, we must have that the matrix
\begin{displaymath}
\begin{bmatrix}
2x_0^2 & a_1 x_{\frac{p+1}{2}}^2 & a_2 x_{1}^2 & \cdots & a_{\frac{p-1}{2}}x_{\frac{p+\frac{p-1}{2}}{2}}^2& \cdots & a_1 x_{\frac{p+(p-1)}{2}}^2 \\
a_1 x_{\frac{p+1}{2}}^2 & 2 x_1^2 & a_1 x_{\frac{p+3}{2}}^2 & \cdots & \cdots &\cdots& a_2 x_0^2 \\
\vdots & \vdots & \ddots
\end{bmatrix}
\end{displaymath}
has rank 1 when $(x_0^2,\ldots,x_{p-1}^2)$ is equal to one of the points found in \ref{lem:fixpoint} (with $\omega$ changed by $\rho$ since $z$ acts on the representation $\mathbb{C}x_0^2 + \ldots +\mathbb{C}x_{p-1}^2$ as $\rho I_p$). These conditions completely determine $a_1,\ldots, a_{\frac{p-1}{2}}$ and therefore we can maximally have $p+1$ algebras isomorphic to $\mathfrak{C}(1:0:\ldots:0:0)$. But lemma \ref{lem:p1elem} gives us that there are minimally $p+1$ points, so we have exactly $p+1$ points. Therefore, the $PSL_2(p)$-orbit of $(1:0:\ldots:0)\in \mathbb{P}^{\frac{p-1}{2}}$ gives all the algebras isomorphic to the quantum space and this orbit consists of $p+1$ points.
\end{proof}
\subsection{Generalizing the duality}
We have a refinement of the duality between the quantum space and $\frac{p-1}{2}$ nonregular algebras, the algebras $\mathfrak{C}(0:0:\ldots:\underbrace{1}_i:\ldots:0), i \neq 0$. This is established in the following way: the algebra $\mathfrak{C}(0:\ldots:\underbrace{1}_i:\ldots:0:0)$ can be recovered from the quantum plane by deleting the $H_p$-orbit of $a_0(x_0x_i + x_i x_0)= a_i x_{\frac{i}{2}}^2$ and instead adding the relations $x_0^2 = x_1^2 = \ldots = x_{p-1}^2=0$. The Koszul dual of this algebra is a commutative algebra with relations $x_0 x_i = x_1 x_{i+1} = \ldots =x_{p-1}x_{p-1+i}=0$. These relations determine $p$ linear spaces of dimension $\frac{p-1}{2}-1$ and for every $0<i\leq \frac{p-1}{2}$ the corresponding union of $p$-dimensional subspaces is different from the other ones. It also follows that the points $(0:0:\ldots:\underbrace{1}_i:\ldots:0)\in \mathbb{P}^{\frac{p-1}{2}}$ are fixed by the element
\begin{displaymath}
\begin{bmatrix}
1 & 0 \\ 1 & 1
\end{bmatrix} \in PSL_2(p).
\end{displaymath}
\par Another duality is given between the $p+1$ hyperplanes in $\mathbb{P}^{\frac{p-1}{2}}$ determined by the $PSL_2(p)$-orbit of the hyperplane $a_0=0$ and the $p+1$ algebras isomorphic to the quantum space. Generically, a point on such a hyperplane determines a unique $H_p$-orbit of $p$ points. These $p$ points are the vertices of the full graph on these $p$ points, which will parametrize the point variety of the corresponding algebra isomorphic to the quantum plane. So we have the following theorem
\begin{theorem}
There is a 1-to-1 correspondence between the algebras isomorphic to the quantum space and the $PSL_2(p)$-orbit of the hyperplane $a_0=0$. This correspondence is given by the following rule: take the unique $p$-Sylow group of $PSL_2(p)$ that fixes the point in the $PSL_2(p)$-orbit of $(1:0:\ldots:0)$. Then there are $\frac{p-1}{2}$ other points that are also fixed by the same subgroup and the hyperplane through these points is the corresponding hyperplane.
\end{theorem}
\begin{proof}
The only thing we need to check is that the action of a $p$-Sylow subgroup does indeed have exactly $p$ fixed points in $\mathbb{P}^{\frac{p-1}{2}}$, which amounts to proving that the matrix associated to a generator of the chosen subgroup has exactly $\frac{p-1}{2}$ different eigenvalues. Since all $p$-Sylow subgroups are conjugated, it suffices to prove this for the subgroup generated by
\begin{displaymath}
M=\begin{bmatrix}
1 & 0 \\ 1 & 1
\end{bmatrix} \in PSL_2(p).
\end{displaymath}
Calculating the action of $e_1e_2$ on an algebra corresponding to the point $(a_1,\ldots,a_\frac{p-1}{2})$, we find that the new generators $y_i$ are equal to 
\begin{displaymath}
y_i = \omega^{-\frac{i(i+1)}{2}}x_i.
\end{displaymath}
If we want to calculate the action of this element on a point $(a_1,\ldots,a_\frac{p-1}{2})$, we find that 
\begin{align*}
y_0 y_i + y_i y_0 &= \omega^{-\frac{i(i+1)}{2}}(x_0 x_i + x_i x_0)\\
                  &= \omega^{-\frac{i(i+1)}{2}} a_i x_{\frac{i}{2}}^2\\
                  &= \omega^{-\frac{i(i+1)}{2}} a_i \omega^{\frac{i}{2}(\frac{i}{2}+1)} y_{\frac{i}{2}}^2 \\
                  &= \omega^{\frac{-i^2}{4}}a_i y_{\frac{i}{2}}^2.
\end{align*}
This makes it clear that the corresponding matrix is a diagonal matrix. Now, if $ \omega^{\frac{-i^2}{4}} =  \omega^{\frac{-j^2}{4}}$, then we must have $i^2 \equiv j^2 \bmod p$ or put differently, $i \equiv \pm j \bmod p$. Since we have $1\leq i,j \leq \frac{p-1}{2}$, this ensures that $i = j$. 
\end{proof}
\subsection{Character series}
In the context of finding noncommutative algebras with properties similar to the polynomial rings, one of the similarities we would like to study are the character series. In this subsection we prove the following theorem.
\begin{theorem}
If the algebra $\mathfrak{C}(a_0:\ldots:a_{\frac{p-1}{2}})$ satisfies the conditions of theorem \ref{th:Reg}, it has the same character series as the polynomial ring in $p$ variables.
\end{theorem}
First, we need some representation theory of $H_p$. More importantly, we need to find  the decompositions of tensor products of simple $p$-dimensional representations. Let $(V,\varphi)$ be the simple representation corresponding to $\varphi(z) = \omega I$ and let $(V_i,\varphi_i)$ be the simple representation defined by $\varphi_i(z) = \omega^i I, i = 1,\ldots p-1$. Let $W_{i,j},\chi_{i,j}$ be the $1$-dimensional representations of $H_p$ such that $\chi_{i,j}(e_1) = \omega^i, \chi_{i,j}(e_2) = \omega^j$ (in particular, $T = \chi_{0,0}$) and define $W = \oplus_{i,j=0}^{p-1} W_{i,j}$. Then character decomposition shows that
\begin{align*}
V \otimes W_{i,j} &= V, \\
V \otimes V_i     &= V_{i+1}^{\oplus p}, i \neq p-1, \\
V \otimes V_{p-1} &= W.
\end{align*}
For example, the tensor algebra $T(V)$ has as character series
\begin{align*}
Ch_{T(V)}(z^k,t) &= \frac{1}{1-\omega^k pt},\\
Ch_{T(V)}(e_1^ke_2^l,t)&=1, (k,l)\neq (0,0).
\end{align*}
To find the character series of our graded Clifford algebras, we will use the fact that such an algebra is a free module of rank $2^p$ over a polynomial ring in $p$ variables, with a basis found by the ordered monomials 
\begin{displaymath}
\{x_{i_1}\ldots x_{i_k}| 0\leq i_1<i_2<\ldots<i_k\leq p-1, 0\leq k \leq p-1\}.
\end{displaymath} So in order to find the character series of the regular Clifford algebras, we first need to find the character series of $S(V)=\mathbb{C}[V]$. Since this is a Koszul algebra, it is more convenient to calculate the character series of $\wedge(V) = (S(V)^!)^*$, for this is a finite dimensional algebra and then use equation \ref{al:chKos}. In the tensor algebra $T(V)$, every degree not divisible by $p$ decomposes as one simple representation with a certain multiplicity. From this we immediately deduce
\begin{align*}
\wedge^0(V) &= T,\\
\wedge^i(V) &= V_i^{\oplus \frac{\binom{p}{i}}{p}},0\neq i \neq p.
\end{align*}
In degree $p$, it is easily checked that $x_0 \wedge \ldots \wedge x_{p-1}$ is fixed by $H_p$ and thus $\wedge^p(V) = T$. Therefore, we have
\begin{align*}
Ch_{\wedge(V)}(z^k,t)&=(1+\omega^k t)^p\\
Ch_{\wedge(V)}(e_1^ke_2^l,t)&=1+t^p, (k,l)\neq (0,0).
\end{align*}
We can now use equation \ref{al:chKos} to find the character series of $S(V)$.
\begin{theorem}
The character series of $S(V)$ is given by 
\begin{align*}
Ch_{S(V)}(z^k,t)&=\frac{1}{(1-\omega^k t)^p}\\
Ch_{S(V)}(e_1^ke_2^l,t)&=\frac{1}{1-t^p}, (k,l)\neq (0,0).
\end{align*}
\end{theorem}
If we want to calculate the character series of the regular Clifford algebra, we use the fact that this algebra is a free module of rank $2^p$ over the polynomial ring $\mathbb{C}[x_0^2,\ldots,x_{p-1}^2]$. The character series of $\mathbb{C}[x_0^2,\ldots,x_{p-1}^2]$ is given by
\begin{align*}
Ch_{\mathbb{C}[x_0^2,\ldots,x_{p-1}^2]}(z^k,t)&=\frac{1}{(1-\omega^{2k} t^2)^p}\\
Ch_{\mathbb{C}[x_0^2,\ldots,x_{p-1}^2]}(e_1^ke_2^l,t)&=\frac{1}{1-t^{2p}}, (k,l)\neq (0,0).
\end{align*}
As an $H_p$-representation, the vector space with basis \begin{displaymath}
\{x_{i_1}\ldots x_{i_k}| 0\leq i_1<i_2<\ldots<i_k\leq p-1, 0\leq k \leq p\}
\end{displaymath} is the same as $\wedge(V)$. Therefore, the character series of the graded Clifford algebra is given by
\begin{align*}
Ch_{\mathfrak{C}(a_0:\ldots:a_{\frac{p-1}{2}})}(z^k,t)&=Ch_{\mathbb{C}[x_0^2,\ldots,x_{p-1}^2]}(z^k,t)Ch_{\wedge(V)}(z^k,t)\\&=\frac{(1+\omega^k t)^p}{(1-\omega^{2k} t^2)^p} = \frac{1}{(1-\omega^k t)^p},\\
Ch_{\mathfrak{C}(a_0:\ldots:a_{\frac{p-1}{2}})}(e_1^ke_2^l,t)&=Ch_{\mathbb{C}[x_0^2,\ldots,x_{p-1}^2]}(e_1^ke_2^l,t)Ch_{\wedge(V)}(e_1^ke_2^l,t)\\&=\frac{1+t^p}{1-t^{2p}}=\frac{1}{1-t^p}, (k,l)\neq (0,0).
\end{align*}
This is indeed the character series of $S(V)$ as expected.
\subsection{The character series of the center}
We again assume that the requirements of theorem \ref{th:Reg} are fulfilled. Since $p\geq 5$ is prime, it is not divisible by 2. This means that the center of the Clifford algebra $\mathfrak{C}(a_0:\ldots:a_{\frac{p-1}{2}})=\mathfrak{C}(a_1,\ldots,a_{\frac{p-1}{2}})$ is a quadratic extension of the polynomial ring $\mathbb{C}[x_0^2,\ldots,x_{p-1}^2]$. The extra element is given by an element of degree $p$, $0 \neq g \in Z(\mathfrak{C}(a_1,\ldots,a_{\frac{p-1}{2}}))_p$ and satisfies the equation $g^2=det(M)$, with $M$ the quadratic form associated to $\mathfrak{C}(a_1,\ldots,a_{\frac{p-1}{2}})$. Since $H_p$ works as algebra automorphisms, it follows that the center is fixed by $H_p$ and so $\mathbb{C}g$ is a 1-dimensional $H_p$-representation. We will now prove
\begin{proposition}
As an $H_p$-representation, $\mathbb{C}g\cong T$.
\end{proposition}
\begin{proof}
The proposition is equivalent with the statement that $\mathbb{C}det(M)\cong T$, so we will prove this fact. Since $g^2=det(M)$, $\mathbb{C}det(M)$ is indeed a $H_p$-representation. Fix a basis for $\mathfrak{C}(a_1,\ldots,a_{\frac{p-1}{2}})_{2p}$ determined by it being a free module over $\mathbb{C}[x_0^2,\ldots,x_{p-1}^2]$ of rank $2^p$ with the ordered monomials as basis and decompose this vector space as simple $H_p$-representations. Calculating the determinant of $M$, we get an element of the center of the form
\begin{displaymath}
det(M)=(2^p+f(a_1,\ldots,a_{\frac{p-1}{2}}))x_0^2x_1^2\ldots x_{p-1}^2 + \ldots
\end{displaymath}
with $2^p+f(a_1,\ldots,a_{\frac{p-1}{2}})\neq 0$ a polynomial (not $0$ because for the quantum plane we have $f(0,\ldots,0) = 0$). For the other elements of our basis, we get coefficients which are polynomials in $a_1,\ldots,a_{\frac{p-1}{2}}$. Because  $2^p+f(a_1,\ldots,a_{\frac{p-1}{2}})\neq 0$, we have a Zariski-open subset of $\mathbb{A}^{\frac{p-1}{2}}$ for which $\mathbb{C}det(M)$ is indeed the trivial representation. Since $\mathbb{C}det(M)$ is always a $H_p$-representation, this means that on the same Zariski-open subset the coefficients of all the nontrivial representations in $\mathfrak{C}(a_1,\ldots,a_{\frac{p-1}{2}})_{2p}$ are 0. Since these coefficients are also polynomial functions on $\mathbb{A}^{\frac{p-1}{2}}$, they must be $0$ on $\mathbb{A}^{\frac{p-1}{2}}$. This means that $\mathbb{C}det(M) \cong T$.
\end{proof}
From this, one easily finds that the character series of $Z(\mathfrak{C}(a_1,\ldots,a_{\frac{p-1}{2}}))$ is given by
\begin{align*}
Ch_{Z(\mathfrak{C}(a_1,\ldots,a_{\frac{p-1}{2}}))}(z^k,t)& = \frac{1+t^p}{(1-\omega^{2k} t^2)^p},\\
Ch_{Z(\mathfrak{C}(a_1,\ldots,a_{\frac{p-1}{2}}))}(e_1^ke_2^l,t)&=\frac{1+t^p}{1-t^{2p}}\\ &= \frac{1}{1-t^p}, (k,l)\neq (0,0).
\end{align*}
\section{Future work}
It is our hope to extend the dualities between the points necessary to make a compatification of $X(5)$ and the points isomorphic to the quantum space to the entire curve $\overline{X(5)}$. Hopefully there will be a subset of $\mathbb{P}^2$ that will determine $5$-dimensional Sklyanin algebras depending on an elliptic curve and a point of order 2. The quadratic form associated to this algebra will be crucial in understanding the representation theory and the underlying $H_5$-action, which will work on the variety that parametrizes representations of dimensions $1,2$ or $4$. Further generalizations to dimension $p$, $p$ being a prime number, will also be worked on in the future.
\bibliographystyle{abbrvnat}
\bibliography{Cliffordalgebrabib}
\nocite{*}
\end{document}